\documentclass[a4paper,10pt]{amsart}

\usepackage[english]{babel}
\usepackage[utf8]{inputenc}

\usepackage{amssymb}
\usepackage{amsmath}
\usepackage{amsthm}
\usepackage{bbm}

\usepackage{enumerate}
\usepackage{tikz}
\usepackage{marginnote}

\newcommand{\bbN}{\mathbb{N}}

\newcommand{\bbR}{\mathbb{R}}

\newcommand{\calJ}{\mathcal{J}}

\newcommand{\calL}{\mathcal{L}}

\newcommand{\calT}{\mathcal{T}}

\renewcommand{\epsilon}{\varepsilon}

\DeclareMathOperator{\one}{\mathbbm{1}}

\newcommand{\argument}{\mathord{\,\cdot\,}}
\newcommand{\dx}{\;\mathrm{d}}

\newcommand{\norm}[1]{\left\lVert #1 \right\rVert}
\newcommand{\modulus}[1]{\left\lvert #1 \right\rvert}

\newcommand{\hmax}{h_{\max}}

\newcommand{\spec}{\sigma}
\newcommand{\Res}{\mathcal{R}}

\theoremstyle{definition}
\newtheorem{definition}{Definition}[section]

\newtheorem{remarks}[definition]{Remarks}

\theoremstyle{plain}

\newtheorem{lemma}[definition]{Lemma}
\newtheorem{theorem}[definition]{Theorem}
\newtheorem{corollary}[definition]{Corollary}

\numberwithin{equation}{section}

\begin{document}

\title[Uniform convergence of stochastic semigroups]{Uniform convergence of stochastic semigroups}

\author{Jochen Gl\"uck}
\address{Jochen Gl\"uck, Fakult\"at f\"ur Informatik und Mathematik, Universit\"at Passau, Innstr.\ 33, 94032 Passau, Germany}
\email{jochen.glueck@uni-passau.de}

\author{Florian G.\ Martin}
\address{Florian G.\ Martin, Institut f\"ur Angewandte Analysis, Universit\"at Ulm, 89069 Ulm, Germany}
\email{florian.martin@alumni.uni-ulm.de}

\subjclass[2010]{47D06; 47D07}

\date{\today}

\begin{abstract}
	For stochastic $C_0$-semigroups on $L^1$-spaces there is wealth of results that show strong convergence to an equilibrium as $t \to \infty$, given that the semigroup  contains a partial integral operator. This has plenty of applications to transport equations and in mathematical biology. However, up to now partial integral operators do not play a prominent role in theorems which yield uniform convergence of the semigroup rather than only strong convergence.
	
	In this article we prove that, for irreducible stochastic semigroups, uniform convergence to an equilibrium is actually equivalent to being partially integral and uniformly mean ergodic. In addition to this Tauberian theorem, we also show that our semigroup is uniformly convergent if and only if it is partially integral and the dual semigroup satisfies a certain irreducibility condition. Our proof is based on a uniform version of a lower bound theorem of Lasota and Yorke, which we combine with various techniques from Banach lattice theory.
\end{abstract}

\maketitle

\section{Introduction and main result} \label{section:introduction}

\subsection*{Strong convergence of partially integral semigroups}

Consider a stochastic $C_0$-semigroup $\calT = (T_t)_{t \in [0,\infty)}$ on $L^1 := L^1(\Omega,\mu)$ over a $\sigma$-finite measure space $(\Omega,\mu)$; by \emph{stochastic} we mean that $T_t f \ge 0$ and $\norm{T_tf} = \norm{f}$ for each time $t \ge 0$ and each function $0 \le f \in L^1$. Such stochastic semigroups occur frequently in models in, for instance, mathematical biology, transport processes and queuing systems. In such models one is often interested in the long-term behaviour of the semigroup $\calT$, and in this context the notions \emph{integral operator} and \emph{partial integral operator} turned out to be very useful:

A bounded linear operator $J: L^1 \to L^1$ is called an \emph{integral operator} if there exists a measurable function $k: \Omega \times \Omega \to \bbR$ such that the following is true for each $f \in L^1$: the function $k(\omega, \argument) f(\argument)$ is in $L^1$ for almost every $\omega \in \Omega$, and the equality
\begin{align*}
	Jf = \int_\Omega k(\argument, \omega)f(\omega) \dx \mu(\omega)
\end{align*}
holds in $L^1$. A positive (hence bounded) linear operator $R: L^1 \to L^1$ is called a \emph{partial integral operator} if there exists a non-zero integral operator $J$ such that $0 \le J \le R$. Sometimes, integral operators are also called \emph{kernel operators} and likewise, partial integral operators are called \emph{partial kernel operators}.

The point is that if one of the operators $T_t$ in the stochastic semigroup $\calT$ is partially integral, this has remarkable consequences for the long-term behaviour of $\calT$. An important result in this direction was proved in 2000 by Pichor and Rudnicki \cite[Theorem~1]{Pichor2000} who showed that, if $\calT$ is irreducible and has a non-zero fixed point, then $T_t$ converges strongly as $t \to \infty$ if at least one of the operators $T_t$ is partially integral. Various modifications and generalisations of this result have appeared since (see for instance \cite{Pichor2016, Pichor2018}). Strong convergence results of this type have proved to be useful in numerous applications, in particular in mathematical biology; see for instance \cite{Banasiak2012, Bobrowski2007, Du2011, Mackey2013, PichorPreprint}, to mention just a few of them. For an application in the theory of kinetic equations we refer to \cite{Mokhtar-Kharroubi2017}.

\subsection*{Uniform convergence}

In the study of kinetic equations, one is often interested in \emph{uniform convergence} (i.e., convergence with respect to the operator norm) rather than merely strong convergence of $T_t$ as $t \to \infty$. As of today, the authors are not aware of any uniform convergence theorem that is specifically designed for semigroups that contain a partial integral operator. In fact, typical results in the study of kinetic equations use compactness properties to show that the \emph{essential growth bound} of $\calT$ is strictly negative; uniform convergence of the semigroup follows then easily from Perron--Frobenius theory. We refer for instance to \cite[Chapter~2]{Mokhtar-Kharroubi1997} and \cite[Sections~3 and~4]{Mokhtar-Kharroubi2014} for a detailed description of such an approach.

Our main result (Theorem~\ref{thm:main}) gives a criterion for uniform convergence that does not rely on knowledge of the essential spectral bound.

\subsection*{Contributions of this article}

The purpose of this article is to prove the following theorem. We recall that a semigroup $\calT = (T_t)_{t \in [0,\infty)}$ (be it strongly continuous or not) on a Banach lattice $E$ is called \emph{irreducible} if the only closed $\calT$-invariant ideals in $E$ are $\{0\}$ and $E$. Irreducibility is equivalent to the condition that, for each non-zero vector $0 \le f \in E$ and each non-zero functional $0 \le \varphi \in E'$ there exists a time $t \ge 0$ such that $\langle \varphi, T_t f \rangle > 0$.

We denote the space of bounded linear operators on a Banach space $E$ by $\calL(E)$.

\begin{theorem} \label{thm:main}
	Let $\calT = (T_t)_{t \in [0,\infty)}$ be a stochastic and irreducible $C_0$-semigroup with generator $A$ on $L^1 := L^1(\Omega,\mu)$ for a $\sigma$-finite measure space $(\Omega,\mu)$. Then the following assertions are equivalent:
	\begin{enumerate}[\upshape (i)]
		\item $T_t$ converges with respect to the operator norm as $t \to \infty$.
		
		\item The number $0$ is a pole of the resolvent of $A$; moreover, there exists a time $t_0 \in [0,\infty)$ and a non-zero integral operator $J \in \calL(L^1)$ such that $0 \le J \le T_{t_0}$.
		
		\item The number $0$ is a pole of the resolvent of $A$; moreover, there exists a time $t_0 \in [0,\infty)$ and a non-zero compact operator $K \in \calL(L^1)$ such that $0 \le K \le T_{t_0}$. 
	\end{enumerate}
\end{theorem}

We prove this theorem in Section~\ref{section:proof-of-the-main-result}. In concrete applications it might sometimes be easier to replace the assumption that $0$ be a pole of the resolvent with an equivalent condition. Let us list several such conditions in the following corollary. For a function $g \in L^\infty(\Omega,\mu)$ we write $g \gg 0$ if there exists $\varepsilon > 0$ such that $g \ge \varepsilon \one$.

\begin{corollary} \label{cor:spectral-condition}
	Let $\calT = (T_t)_{t \in [0,\infty)}$ be a stochastic and irreducible $C_0$-semigroup with generator $A$ on $L^1(\Omega,\mu)$ for a $\sigma$-finite measure space $(\Omega,\mu)$. If there exists a time $t_0 \in [0,\infty)$ such that $T_{t_0}$ dominates a non-zero positive integral operator or a non-zero positive compact operator, then the following assertions are equivalent:
	\begin{enumerate}[\upshape (i)]
		\item $T_t$ converges with respect to the operator norm as $t \to \infty$.
		
		\item The Ces\`{a}ro means $C_t := \frac{1}{t} \int_0^t T_s \dx s$ converge uniformly as $t \to \infty$ (i.e., $\calT$ is uniformly mean ergodic).
		
		\item The number $0$ is a pole of the resolvent $\Res(\argument,A)$.
		
		\item The dual semigroup $\calT' := (T_t')_{t \in [0,\infty)}$ on $L^\infty(\Omega,\mu)$ is irreducible.
		
		\item For each non-zero $0 \le f \in L^\infty(\Omega,\mu)$ there exists a number $\lambda > 0$ such that $\Res(\lambda,A)'f \gg 0$.
		
		\item For each non-zero $0 \le f \in L^\infty(\Omega,\mu)$ and each number $\lambda > 0$ we have $\Res(\lambda,A)'f \gg 0$.
	\end{enumerate}
\end{corollary}

The integral in the definition of the Ces\`{a}ro means $C_t$ is meant in the strong sense. We point out that the dual semigroup $\calT'$ need not be $C_0$, in general (in fact, there are no $C_0$-semigroups on an $L^\infty$-space except for those with bounded generator \cite[Theorem~A-II-3.6]{Arendt1986}). Moreover, we note that irreducibility is a rather strong property on $L^\infty(\Omega,\mu)$ because this space contains a lot of closed ideals. If $L^1(\Omega,\mu)$ is separable, then compact operators on $L^1(\Omega,\mu)$ are automatically integral operators \cite[Theorem~10 on p.~507]{Dunford1958}, so that the assumption in Corollary~\ref{cor:spectral-condition} can than be slightly simplified.

We defer the proof of Corollary~\ref{cor:spectral-condition} to Section~\ref{section:proof-of-spectral-proposition}.

\begin{remarks}
	\begin{enumerate}[(a)]
		\item The implication ``(ii) $\Rightarrow$ (i)'' in Corollary~\ref{cor:spectral-condition} is a Tauberian theorem since it concludes convergence of the semigroup from convergence of its Ces\`aro means.
		
		\item Theorem~\ref{thm:main} can in particular be applied to stochastic semigroups on $\ell^1$ since every bounded linear operator on $\ell^1$ is an integral operator; so if $0$ is a pole of the resolvent for such a semigroup, then the semigroup converges uniformly as $t \to \infty$.
		
		This observation is not new, though; it can easily be derived from a classical theorem of Williams \cite[Assertion~1 of the Theorem in Section~2]{Williams1967}.
		
		\item The most interesting interesting implications in the theorem are certainly those from~(ii) or~(iii) to~(i). Concerning these implications, a few remarks are in order, so assume that~(i) or~(ii) holds. Then it follows from a Niiro--Sawashima type result that all spectral values of $A$ on the imaginary axis in fact poles of the resolvent $\Res(\argument,A)$ \cite[Theorem~C-III-3.12]{Arendt1986} and thus, in particular, eigenvalues. But due to the domination of a non-zero integral operator or a non-zero compact operator, $A$ cannot have eigenvalues in $i\bbR \setminus \{0\}$ \cite[Theorem~5.4]{Gerlach2019}. Hence, $\spec(A) \cap i\bbR = \{0\}$.
		
		So the main point Theorem~\ref{thm:main} is not the triviality of the peripheral spectrum of $A$; instead, it is (a) that we do not know a priori that $\spec(A) \setminus \{0\}$ is bounded away from the imaginary axis and (b) that we do not have a spectral mapping theorem available which could be used to derive the behaviour of the semigroup from the behaviour of the spectrum.
	\end{enumerate}
\end{remarks}

\subsection*{On (non-)$\sigma$-finite measure spaces}

Throughout the paper we assume all occurring measure spaces to be $\sigma$-finite -- but this is merely a matter of notational convenience rather than of theoretical necessity. In fact, much of what we have to say can also be formulated, and proved, on so-called \emph{AL-Banach lattices} (which are essentially the same thing as $L^1$-spaces over arbitrary measure spaces). But in order to do so one needs to be careful to use the appropriate definitions of all objects that occur in our treatment. To give just one example (there are several of them), one has to replace $L^\infty(\Omega,\mu)$ with the dual space of $L^1(\Omega,\mu)$, since those spaces do not necessarily coincide in the non-$\sigma$-finite case.

\subsection*{Other spaces than $L^1$}

Strong convergence results for partially integral semigroups hold in more general spaces than $L^1$. For instance, on Banach lattices with order continuous norm (which include, in particular, all $L^p$-spaces for $1 \le p < \infty$), similar results as on $L^1$ can be proved (see \cite[Theorem~4.2]{Gerlach2013} and \cite[Theorem~2]{Gerlach2017}), and in fact, one can even dispense with the time continuity assumption on the semigroup, see \cite[Section~4.2]{Gerlach2019} and \cite{Glueck2019}. Finally, it is also worthwhile to mention the special case where the semigroup consists entirely of integral operators; this situation occurs in the study of partial differential equations (see \cite[Section~5]{Arendt2008}), and convergence results for this case can be found in \cite[Korollar~3.11]{Greiner1982a}, \cite[Theorem~1]{Gerlach2017} and \cite[Section~4.1]{Gerlach2019}.

For uniform convergence the situation is less clear, though: we do not know whether our main results remains true on, say, $L^p$-spaces for $p \in (1,\infty)$, and even on $L^1$ we do not know whether one can drop the condition that the semigroup be stochastic. Our proofs make heavy use of both the $L^1$-structure of the space and of the assumption that the semigroup be stochastic.

\subsection*{Prerequisites}

We assume the reader to be familiar with the basic theory of Banach lattices (see for instance \cite{Schaefer1974} or \cite{Meyer-Nieberg1991}) and with standard $C_0$-semigroup theory (see for instance \cite{Pazy1983} or \cite{Engel2000}). We recall a few basic notions, though, as we proceed. Let us point out that we use the standard terminology from Banach lattice theory where \emph{positive} always means ``$\ge 0$'' (i.e., positivity does not mean a strict inequality in any sense).

\subsection*{Organisation of the article}

The main purpose of this article is to prove Theorem~\ref{thm:main}. In order to keep the article short and concise, we shall not discuss any applications in detail here. At the end of the article (Section~\ref{section:outlook}) we give a brief outlook, though, where we outline several potential applications.

The proof of the implication ``(iii) $\Rightarrow$ (i)'' in Theorem~\ref{thm:main} relies on a uniform version of a lower bound convergence theorem which was, in \cite[Corollary~3.6]{GerlachLB}, derived from a strong convergence version of the same result by means of an ultrapower technique. Since this uniform lower bound theorem is an important ingredient in the proof of our main result, we include a more elementary proof of it in Section~\ref{section:a-lower-bound-theorem}. The proof of Theorem~\ref{thm:main} is given in Section~\ref{section:proof-of-the-main-result}, and the proof of Corollary~\ref{cor:spectral-condition} can be found in Section~\ref{section:proof-of-spectral-proposition}.

\section{A lower bound theorem} \label{section:a-lower-bound-theorem}

Let us consider a stochastic semigroup $\calT = (T_t)_{t \in (0,\infty)}$ on $L^1 := L^1(\Omega,\mu)$ over a $\sigma$-finite measure space $(\Omega,\mu)$. A function $0 \le h \in L^1$ is called a \emph{lower bound} for $\calT$ if, for every function $0 \le f \in L^1$ of norm $1$, the orbit $(0,\infty) \ni t \mapsto T_t f \in L^1(\Omega,\mu)$ \emph{asymptotically dominates} $h$, meaning that $\norm{(T_t f - h)^-} \to 0$ as $t \to \infty$. It was first shown by Lasota and Yorke \cite[Theorem~2 and Remark~3]{Lasota1982} that $T_t$ converges strongly to a rank-$1$ projection as $t \to \infty$ if and only if there exists a non-zero lower bound for the semigroup $\calT$. For the proof of our main theorem we need a uniform version of this result; to this end, we introduce the following terminology.

\begin{definition} \label{def:uniform-lower-bounds}
	Let $J = \bbN_0$ or $J = [0,\infty)$ and let $\calT = (T_t)_{t \in J}$ be a stochastic operator semigroup on $L^1 := L^1(\Omega,\mu)$ over a $\sigma$-finite measure space $(\Omega,\mu)$. Let $0 \le h \in L^1$. If
	\begin{align*}
		\sup\left\{ \norm{(T_t f - h)^-}: \; 0 \le f \in L^1, \; \norm{f} = 1 \right\} \to 0 \qquad \text{as} \qquad t \to \infty,
	\end{align*}
	then the function $h$ is called a \emph{uniform lower bound} for $\calT$.
\end{definition}

Here is a uniform version of the Lasota--Yorke lower bound theorem.

\begin{theorem} \label{thm:lasota-yorke-uniform}
	Let $\calJ = \bbN_0$ or $\calJ = [0,\infty)$ and let $\calT = (T_t)_{t \in \calJ}$ be a stochastic operator semigroup on $L^1 := L^1(\Omega,\mu)$ over a $\sigma$-finite measure space $(\Omega,\mu)$. The following assertions are equivalent:
	\begin{enumerate}[\upshape (i)]
		\item $T_t$ converges with respect to the operator norm as $t \to \infty$ and the limit operator has rank $1$.
		\item There exists a non-zero uniform lower bound $0 \le h \in L^1$ for $\calT$.
	\end{enumerate}
\end{theorem}

In case that $J = [0,\infty)$, this theorem does not require any time regularity of the semigroup. Theorem~\ref{thm:lasota-yorke-uniform} was proved -- in a slightly more general version for semigroups which are merely positive and bounded, and over general measure spaces -- in \cite[Corollary~3.6]{GerlachLB} by M.\ Gerlach and the first-named of the present authors. There, the theorem was reduced to the classical theorem of Lasota and Yorke (which characterises strong convergence) by means of an ultrapower argument. For the convenience of the reader and in order to be more self-contained, we include a more direct proof of Theorem~\ref{thm:lasota-yorke-uniform} here which is quite close in spirit to the original proof of Lasota and Yorke. Our presentation follows mainly \cite[Theorem~3.2]{GerlachLB}, although we need certain estimates to be uniform in the present setting. At first we show the following lemma:

\begin{lemma} \label{lemma:maximal-uniform-lower-bound}
	In the situation of Theorem \ref{thm:lasota-yorke-uniform} assume that condition (ii) is satisfied and let $\emptyset \neq H \subseteq (L^1)_+$ denote the set of all uniform lower bounds for $\calT$. Then $H$ has a largest element $\hmax$; this element satisfies $0 < \norm{\hmax} \le 1$ and $T_t \hmax = \hmax$ for all $t \in \calJ$.
\end{lemma}

For the proof we note that the norm on $L^1$ is \emph{strictly monotone}, meaning that $\norm{f} < \norm{g}$ whenever $0 \le f \le g$ and $f \not= g$.

\begin{proof}[Proof of Lemma~\ref{lemma:maximal-uniform-lower-bound}]
	One readily checks that $H$ is closed and as $\calT$ is stochastic we obtain $\norm{h} \le 1$ for all $h \in H$. Furthermore, the relation $(a - b \vee c)^- = (a - b)^- \vee (a - c)^- \le (a - b)^- + (a - c)^-$ for any three functions $a,b,c \in L^1$ implies that $h_1 \vee h_2 \in H$ for all $h_1, h_2 \in H$. Hence, $(h)_{h \in H} \subseteq H$ is an increasing and norm-bounded net. Using the fact that the $L^1$-norm is additive on the positive cone,
it is easy to see that $(h)_{h \in H}$ is a Cauchy net and thus convergent in $H$. Clearly, $\hmax := \lim_{h \in H} h \in H$ is the largest element of $H$ and satisfies $0 < \norm{\hmax} \le 1$ as asserted. It remains to prove that $T_t \hmax = \hmax$ for all $t \in \calJ$. It is easy to verify that $H$ is invariant under $\calT$ and hence $T_t \hmax \in H$ for all $t \in \calJ$. Since $\hmax$ is the largest element of $H$ we thus conclude that $T_t \hmax \le \hmax$, and since $\calT$ is stochastic and the $L^1$-norm is strictly monotone we obtain $T_t \hmax = \hmax$ for all $t \in \calJ$.
\end{proof}

Now we can prove Theorem~\ref{thm:lasota-yorke-uniform}. By $\one \in L^\infty := L^\infty(\Omega,\mu)$ we denote the constant function with value $1$.

\begin{proof}[Proof of Theorem \ref{thm:lasota-yorke-uniform}]
	(i) $\Rightarrow$ (ii): Assume that $T_t$ converges with respect to the operator norm to a limit operator $P$ of rank $1$. The operator $P$ is stochastic, so it is given by $Pf = \langle \one,f\rangle f_0$ for all $f \in L^1$, where $f_0 \ge 0$ is a function in $PL^1$ of norm $1$. In particular, we have $Pf=f_0$ for every normalized function $f\in (L^1)_+$ and thus
\[\sup\left\{ \norm{(T_t f - f_0)^-}: \; 0 \le f \in L^1, \; \norm{f} = 1 \right\} \le \norm{T_t-P} \to 0 \qquad \text{as} \qquad t \to \infty.\] In other words, $f_0$ is a non-zero uniform lower bound for $\calT$, which shows (ii).

	(ii) $\Rightarrow$ (i): Assume that (ii) holds. By Lemma \ref{lemma:maximal-uniform-lower-bound} there exists a largest uniform lower bound $\hmax$ for $\calT$, and the function $\hmax$ is non-zero and satisfies $T_t \hmax = \hmax$ for all $t \in \calJ$. In the following we show that $\norm{\hmax} = 1$. Assume to the contrary that there exists a number $\delta \in (0,1)$ such that $\norm{\hmax} = 1 - \delta$. We will prove that $(1 + \delta) \hmax$ is also a uniform lower bound for $\calT$, which contradicts the fact that $\hmax$ is the largest element of $H$. 
	
	So let $\epsilon > 0$. Since $\hmax$ is a uniform lower bound there exist a time $t_0 \in \calJ$ and vectors $e_f \ge 0$ of norm $\norm{e_f} < \epsilon$ such that $T_{t_0}f + e_f \ge \hmax$ for every $f \in (L^1)_+$ of norm $1$. For each normalized positive function $f$ we define $g_f := T_{t_0}f - \hmax + e_f \ge 0$; the norm of those vectors $g_f$ can be estimated by
	\begin{align*}
		\norm{g_f} = \langle \one, T_{t_0}f\rangle - \langle \one, \hmax\rangle + \langle \one, e_f\rangle \ge \norm{T_{t_0}f} - \norm{\hmax} = 1 - \norm{\hmax} = \delta.
	\end{align*}
	Therefore, we have $\sup_{f} \norm{(T_t \frac{g_f}{\delta} - \hmax)^-} \to 0$ as $t \to \infty$, and consequently, there exist another time $t_1 \in \calJ$ and vectors $\tilde{e}_f \ge 0$ of norm $\norm{\tilde{e}_f} < \epsilon$ such that $T_t g_f + \tilde{e}_f \ge \delta \hmax$ for all $t \in \calJ$ with $t \ge t_1$ and for all normalized functions $0 \le f \in L^1$.
	
	Now, fix a normalized function $0 \le f \in L^1$ and a time $t \ge t_1$ in $\calJ$. Using the definition of $g_f$ and the $\calT$-invariance of $\hmax$ we obtain
	\begin{align*}
		T_t T_{t_0}f = T_t g_f - T_t e_f + \hmax \ge (1 + \delta)\hmax - \tilde{e}_f - T_t e_f.
	\end{align*}
	We have $\norm{\tilde{e}_f + T_t e_f} < 2 \epsilon$, so we have shown that $\norm{\big(T_tf - (1 + \delta)\hmax \big)^-} < 2 \epsilon$ for all $t \in \calJ$ with $t \ge t_1+t_0$. Since $t_0$ and $t_1$ do not depend on the choice of $f$, it follows that $(1 + \delta)\hmax$ is indeed a uniform lower bound for $\calT$. As this contradicts the maximality of $\hmax$, we conclude that $\norm{\hmax} = 1$.

	Finally we prove that $T_t$ operator norm converges to the rank-$1$ operator $P$ defined by $Pg := \langle \one,g\rangle \hmax$ for all $g \in L^1$. For every normalized $0 \le f \in L^1$ and every time $t$ we have
	\begin{align*}
		& \norm{(T_t f - \hmax)^+} - \norm{(T_tf - \hmax)^-} = \langle \one, (T_tf - \hmax)^+ - (T_tf - \hmax)^-\rangle \\
		& = \langle \one, T_tf - \hmax\rangle = \norm{T_t f} - \norm{\hmax} = 0,
	\end{align*}
	as both $T_tf$ and $\hmax$ are of norm $1$; hence, $(T_tf - \hmax)^+$ and $(T_tf - \hmax)^-$ have the same norm, so the norm of $T_tf - \hmax$ equals twice the norm of $(T_tf - \hmax)^-$. Using that $Pf = \hmax$ for every normalized $0 \le f \in L^1$, we thus obtain
	\begin{align*}
		& \sup\left\{ \norm{T_t f - Pf}: \; 0 \le f \in L^1, \; \norm{f} = 1 \right\} \\
		& = 2 \sup\left\{ \norm{(T_t f - \hmax)^-}: \; 0 \le f \in L^1, \; \norm{f} = 1 \right\} \to 0
	\end{align*}
	as $t \to \infty$. From this, one easily derives that $T_t$ converges to $P$ with respect to the operator norm as $t \to \infty$.
\end{proof}

Theorems which derive convergence of a semigroup from the existence of lower bounds exist in a wide range of variations; for further results in $L^1$-spaces we refer, for instance, to \cite{Zalewska-Mitura1994, Ding2003, GerlachLB}, and for results on non-commutative $L^1$-spaces, or more generally on abstract state spaces, we refer for instance to \cite[Section~3.3]{Emelyanov2007} and \cite{Glueck2019a}. To the best of our knowledge, the only known uniform version of a lower bound convergence theorem so far is \cite[Corollary~3.6]{GerlachLB}, which we reproved (in the important special case of stochastic semigroups) in Theorem~\ref{thm:lasota-yorke-uniform} above.

\section{Proof of the main result} \label{section:proof-of-the-main-result}

The purpose of the section is to prove Theorem~\ref{thm:main}. We will need a few facts from Banach lattice theory in the proof. 

First, we note that the operator space $\calL(L^1)$ is itself a Banach lattice with respect to the operator norm; this is due to the special structure of $L^1$-spaces \cite[Theorem~IV.1.5]{Schaefer1974}. Moreover, we will need that an operator $I \in \calL(L^1)$ is an integral operator if and only if it is an element of the band in $\calL(L^1)$ that is generated by the finite-rank operators \cite[Proposition~IV.9.8]{Schaefer1974}.

We recall that a vector $g$ in a Banach lattice $E$ is called a \emph{quasi-interior point} if $g$ is positive and the principal ideal generated by $g$ is dense in $E$ (see for instance \cite[Section~II.6]{Schaefer1974} for a more detailed discussion). If $E = L^1$, then $g$ is a quasi-interior point if and only if $g$ is strictly positive almost everywhere; if $E = L^\infty$, then $g$ is a quasi-interior point if and only if there exists a number $\varepsilon > 0$ such that $g \ge \varepsilon \one$ (i.e., $g \gg 0$ in the notation introduced before Corollary~\ref{cor:spectral-condition}). One important property of quasi-interior points is that, if $g \in E$ is a quasi-interior point and $R \in \calL(E)$ is positive and non-zero, then $Rg \not= 0$.

\begin{proof}[Proof of Theorem~\ref{thm:main}]
	We first note that $0$ is a spectral value of $A$ since $\calT$ is stochastic. Moreover, $0$ is a pole of the resolvent $\Res(\argument,A)$ if and only if the Ces\`aro means of the semigroup converge uniformly as time tends to $\infty$; this follows from \cite[Theorem~V.4.10]{Engel2000}.
	
	``(i) $\Rightarrow$ (ii)'' If $T_t$ converges uniformly to an operator $P \in \calL(L^1)$ as $t \to \infty$, then so do the Ces\`aro means of the semigroup; hence, we conclude that $0$ is a pole of the resolvent of $A$. As $\calT$ is irreducible, the limit projection $P$ has rank-$1$ and is thus an integral operator; see for instance \cite[Proposition~C-III-3.5]{Arendt1986}. Now we use that $\calL(L^1)$ is a Banach lattice when endowed with the operator norm: this implies that
	\begin{align*}
		T_t \land P \to P
	\end{align*}
	with respect to the operator norm as $t \to \infty$. In particular, there exists a time $t_0 \in [0,\infty)$ such that $J := T_{t_0} \land P \not= 0$. Since the integral operators are a band in $\calL(L^1)$, $J$ is an integral operator, and obviously $0 \le J \le T_{t_0}$.
	
	``(ii) $\Rightarrow$ (iii)'' We show that there exists a time $t_1 \in (0,\infty)$ and a non-zero compact operator $K$ such that $0 \le K \le T_{t_1}$.	
	
	Since $J$ is an integral operator, there exists a finite-rank operator $F \in \calL(L^1)$ which is not disjoint to $J$. Moreover, the operator modulus $\modulus{F}$ of $F$ is dominated by a positive finite-rank operator $G \in \calL(L^1)$, and so we have $J \land G \not= 0$. The operator $J \land G$ itself might not be compact, but its square is compact according to \cite[Corollary~3.7.15(ii)]{Meyer-Nieberg1991}; since we do not know that $(J \land G)^2$ is non-zero though, we employ the following construction:
	
	Choose a positive function $f \in L^1$ such that $(J \land G) f \not= 0$. Since our semigroup $\calT$ is irreducible, the resolvent of its generators at the point $1$ sends every non-zero positive vector to a quasi-interior point of $(L^1)_+$, i.e., the function $\Res(1,A)(J \land G) f$ is strictly positive almost everywhere \cite[p.~306]{Arendt1986}. Using again that $J \land G$ is non-zero, this in turn implies that $(J \land G) \Res(1,A) (J \land G)f \not= 0$. The Laplace transform representation of $\Res(1,A)$ now implies that there exists a time $s \in [0,\infty)$ such that $(J \land G) T_s (J \land G) f \not= 0$. Again due to the irreducibility of $\calT$ we have $T_tg \not= 0$ for every non-zero vector $g \in (L^1)_+$ and for every time $t \in [0,\infty)$; see for instance \cite[Theorem~C-III-3.2(a)]{Arendt1986}. Hence, the operator $K := \left(T_s (J \land G) \right)^2$ sends $f$ to a non-zero vector and is thus non-zero.
	
	Since $T_s (J \land G)$ is dominated by the compact operator $T_s G$, it also follows that $K$ is compact \cite[Corollary~3.7.15(ii)]{Meyer-Nieberg1991}. Finally, we clearly have $0 \le K \le T_{2(s+t_0)}$, so the proof of this implication is complete with $t_1 := 2(s+t_0)$.	
	
	``(iii) $\Rightarrow$ (i)'' Let us denote the Ces{\`a}ro means of the semigroup by $C_t$, i.e., we set $C_tf := \frac{1}{t} \int_0^t T_s f \dx s$ for each $f \in L^1$ and each $t \in (0,\infty)$. As $0$ is a (first order) pole of the resolvent, $C_t$ converges with respect to the operator norm to an operator $P \in \calL(L^1)$ as $t \to \infty$.

	We note that the operator $P$ is stochastic and a projection onto the fixed space of $\calT$; moreover, $P$ coincides with the spectral projection of $A$ associated with the pole $0$. From \cite[Proposition~C-III.3.5(a) and~(d)]{Arendt1986} we obtain that
	\begin{align*}
		P = \one \otimes g,
	\end{align*}
	where $g \in (L^1)_+$ is a fixed point of $\calT$ and a quasi-interior point in $L^1$.
	
	As follows from \cite[Theorem~3.11]{Gerlach2019}, the assumptions on our semigroup together with the existence of the quasi-interior fixed point $g$ imply \emph{strong} convergence of $T_t$ as $t \to \infty$ (note that assumption~(a) in \cite[Theorem~3.11]{Gerlach2019} is satisfied since the norm on $L^1$ is strictly monotone, see \cite[Proposition~3.13(a)]{Gerlach2019}; alternatively, one can derive the validity of this assumption from the irreducibility of the semigroup $\calT$, see \cite[Proposition~3.13(c)]{Gerlach2019}). Clearly, the strong limit of $T_t$ as $t \to \infty$ coincides with $P$.
	
	We need to show is that the convergence of $T_t$ to $P$ actually takes place with respect to the operator norm. To this end, let us decompose each operator $T_t$ for $t \ge t_0$ into a sum of a positive compact operator and another positive operator. For the time $t_0$ we can simply set $R := T_{t_0} - K$ and thus obtain the decomposition $T_{t_0} = K + R$, where $K$ is positive and compact and where $R$ is positive. For $t > t_0$ we define
	\begin{align*}
		K_t := K T_{t - t_0} 
		\qquad \text{and} \qquad
		R_t := R T_{t - t_0}
	\end{align*}
	Then we indeed have $T_t = K_t + R_t$ and $K_t,R_t \ge 0$ for all times $t \in (t_0,\infty)$, and each operator $K_t$ is compact. 
	
	The main idea is now as follows: if we could find a time $t_1 > t_0$ for which we have $\norm{R_{t_1}} < 1$, then $T_{t_1}$ would be closer than $1$ to the compact operator, i.e., $T_{t_1}$ would be \emph{quasi-compact}; operator norm convergence of the semigroup would thus follow from classical Perron--Frobenius theory, see for instance \cite[Theorem~4]{Lotz1986} or \cite[Theorem~C-IV-2.1]{Arendt1986}.
	
	Under the given assumptions, it seems unclear how to show that $\norm{R_{t_1}} < 1$ for some time $t_1$, but we can prove a slightly weaker assertion, namely that there is a time $t_1 > t_0$ such that the product of $R_{t_1}$ and the Ces{\`a}ro mean $C_{t_1}$ has norm $<1$. To do so, we first observe that the identities
	\begin{align}
		\label{eq:property-of-decomposition}
		K_t T_s = K_{t+s} \qquad \text{and} \qquad R_t T_s = R_{t+s}
	\end{align}
	hold for all $t \in (t_0,\infty)$ and all $s \in [0,\infty)$; this will be useful in the sequel.
	
	Since the fixed point $g$ of $\calT$ is a quasi-interior point of $(L^1)_+$ and since $K$ is non-zero, we have $Kg \not= 0$. Let us set $\delta := \norm{Kg} > 0$. As indicated above, we now show that there exists a time $t_1 > t_0$ such that $\norm{R_{t_1} C_{t_1}} \le 1 - \delta/2$.
	
	Indeed, since $C_t$ converges with respect to the operator norm to $P = \one \otimes g$ as $t \to \infty$, there exists a time $t_1 > t_0$ for which we have $\norm{C_{t_1} - P} \le \delta / 2$. The operator $R_{t_1}$ is contractive (as it is positive and dominated by $T_{t_1}$), so the distance of $R_{t_1} C_{t_1}$ to $R_{t_1}P$ is not larger than $\delta/2$, either. On the other hand, the latter operator is given by
	\begin{align*}
		R_{t_1}P = \one \otimes R_{t_1}g = \one \otimes RT_{t_1-t_0}g = \one \otimes Rg,
	\end{align*}
	and thus has the norm
	\begin{align*}
		\norm{R_{t_1}P} = \norm{Rg} = \langle \one, T_{t_0}g\rangle - \langle \one, Kg\rangle = \norm{g} - \norm{Kg} = 1 - \delta.
	\end{align*}
	Thus, the norm of $R_{t_1}C_{t_1}$ is indeed not larger than $1-\delta/2$.
	
	Next we use this to estimate $\norm{R_sf}$ for every normalised $f \ge 0$ and for a certain ($f$-dependent) time $s$: for every vector $0 \le f \in L^1$ of norm $1$ we compute
	\begin{align*}
		1 - \frac{\delta}{2} & \, \ge \, \norm{R_{t_1} C_{t_1}f} = \norm{\frac{1}{t_1} \int_{0}^{t_1} R_{t_1}T_s f \dx s } \\
		& = \, \frac{1}{t_1} \int_0^{t_1} \norm{R_{t_1+s} f} \dx s = \frac{1}{t_1} \int_{t_1}^{2t_1} \norm{R_s f} \dx s;
	\end{align*}
	for the equality between both lines we used~\eqref{eq:property-of-decomposition} and the fact that the norm is additive on the positive cone of $L^1$. Hence, there exists a time $s_f \in [t_1,2t_1]$ such that $\norm{R_{s_f}f} \le 1 - \delta/2$.
	
	If $s_f$ was independent of $f$, then we would have found a time $s := s_f$ for which $\norm{R_s} < 1$, and the proof would be complete; but of course, this does not work, since $s_f$ is actually dependent on $f$. Fortunately though, each time $s_f$ is bounded from above by $2t_1$, and this will be sufficient to show now that the vector $\frac{\delta}{2}g$ is a non-zero uniform lower bound of our semigroup. Operator norm convergence of the semigroup will thus follow from Theorem~\ref{thm:lasota-yorke-uniform}.
	
	So fix $\varepsilon > 0$ and $0 \le f \in L^1$. Since $K$ is compact, the convergence of $T_t$ to $P$ as $t\to \infty$ is uniform on the image of the positive unit sphere of $L^1$ under $K$, i.e., there is a time $t_2 > 0$ such that $\norm{T_tK \tilde f - PK\tilde f} < \varepsilon$ for each $t \ge t_2$ and each $\tilde f \in (L^1)_+$ of norm $1$.
	
	We now show that, for $t \ge 2t_1 + t_2$, the negative part of $T_t f - \frac{\delta}{2}g$ has norm smaller than $\varepsilon$, which completes the proof. To this end, we define the vector $z_{f,t} := T_{t-s_f}KT_{s_f-t_0}f$ and make the following three observations which are valid for all times $t \ge 2t_1 + t_2$:
	\begin{enumerate}[\upshape (a)]
		\item We have $T_tf \ge z_{f,t}$. Indeed,
			\begin{align*}
				T_t f = T_{t-s_f}T_{s_f} f \ge T_{t-s_f} K_{s_f} f = z_{f,t}.
			\end{align*}
		\item The vector $z_{f,t}$ is closer than $\varepsilon$ to the vector $PKT_{s_f-t_0}f$. This follows since $t-s_f \ge t_2$ and since the vector $T_{s_f-t_0}f$ is positive and of norm $1$.
		\item The vector $PKT_{s_f-t_0}f$ dominates the vector $\frac{\delta}{2}g$. Indeed, we have
			\begin{align*}
				PKT_{s_f-t_0}f = (\one \otimes g) K_{s_f}f = \langle \one, K_{s_f}f \rangle g = \norm{K_{s_f}f}g,
			\end{align*}
			and the norm of $K_{s_f}f$ is at least $\delta/2$ since we have proved above that the norm of $\norm{R_{s_f}f}$ is no more than $1 - \delta/2$.
	\end{enumerate}
	Observations~(a)--(c) prove that we indeed have $\norm{\left( T_t f - \frac{\delta}{2}g \right)^-} < \varepsilon$ for each time $t \ge 2t_1 + t_2$, so $\frac{\delta}{2}g$ is a uniform lower bound for our semigroup, as claimed.
\end{proof}

\section{Proof of Corollary~\ref{cor:spectral-condition}} \label{section:proof-of-spectral-proposition}

The following proof is a combination of Theorem~\ref{thm:main}, several known results from the literature and standard arguments from operator theory on Banach lattices.

\begin{proof}[Proof of Corollary~\ref{cor:spectral-condition}]
	We show the implications ``(i) $\Leftrightarrow$ (iii)'', ``(ii) $\Leftrightarrow$ (iii)'' and ``(i) $\Rightarrow$ (iv) $\Rightarrow$ (vi) $\Rightarrow$ (v) $\Rightarrow$ (vi) $\Rightarrow$ (iii)''.
	
	``(i) $\Leftrightarrow$ (iii)'' This is part of Theorem~\ref{thm:main}.
	
	``(ii) $\Leftrightarrow$ (iii)'' Since the constant function $\one$ is a fixed point of the dual semigroup $\calT'$, it follows that $0$ is a spectral value of $A$. Moreover, if $0$ is a pole of the resolvent $\Res(\argument,A)$, then it is of order $1$ since the semigroup $\calT$ is bounded. Hence, the equivalence of~(i) and~(ii) follows from \cite[Theorem~V.4.10]{Engel2000}.
	
	``(i) $\Rightarrow$ (iv)'' Denote by $P \in \calL(L^1)$ the limit of $T_t$ as $t \to \infty$. Then $P$ is a stochastic operator and the spectral projection associated with the pole $0$ of $A$. The irreducibility of $\calT$ thus implies that $P$ is of the form $P = \one \otimes g$, where $0 \le g \in L^1$ is a quasi-interior point \cite[Proposition~C-III.3.5(a) and~(d)]{Arendt1986}.
	
	Now, let $I \subseteq L^\infty(\Omega,\mu)$ be a non-zero closed ideal which is invariant under the dual semigroup $\calT'$, and choose a non-zero vector $0 \le f \in I$. Since $T_t'$ converges to $P'$ with respect to the operator norm, we conclude that $\langle g, f\rangle \one = P'f \in I$. As $\langle g, f\rangle > 0$, it follows that $\one \in I$, so $I = L^\infty$.
	
	``(iv) $\Rightarrow$ (vi)'' Fix a non-zero function $0 \le f \in L^\infty$ and a number $\lambda \in (0,\infty)$. Using the Laplace transform representation of the resolvent, one can easily show that $T_t \Res(\lambda,A) \le e^{\lambda t} \Res(\lambda,A)$ for each $t \ge 0$, and hence the same inequality holds for the dual operators. So we have
	\begin{align*}
		T_t' \Res(\lambda,A)'f \le e^{\lambda t} \Res(\lambda,A)'f
	\end{align*}
	for each $t \ge 0$, which shows that the principal ideal in $L^\infty$ generated by the function $\Res(\lambda,A)'f$ is invariant under $\calT'$; hence, so its closure $I$. As $I$ contains the vector $\Res(\lambda,A)'f$, which is non-zero due to the injectivity of the operator $\Res(\lambda,A)'$, it follows from the irreducibility of $\calT'$ that $I = L^\infty$. This proves that $\Res(\lambda,A)'f$ is a quasi-interior point in $L^\infty$.
	
	``(vi) $\Rightarrow$ (v)'' This is obvious.
	
	``(v) $\Rightarrow$ (vi)'' Fix $\lambda > 0$ and a non-zero function $0 \le f \in L^\infty$, and let $I \subseteq L^\infty$ denote the closure of the principal ideal generated by $\Res(\lambda,A)'f$. We have to prove that $I = L^\infty$.
	
	For each $\mu > \lambda$ it follows from the resolvent identity that
	\begin{align*}
		0 \le \Res(\mu,A)'\Res(\lambda,A)'f = \frac{1}{\mu - \lambda}\Res(\lambda,A)'f - \frac{1}{\mu - \lambda} \Res(\mu,A)'f \le \frac{1}{\mu - \lambda} \Res(\lambda,A)'f.
	\end{align*}
	So the principal ideal generated by $\Res(\lambda,A)'f$ is invariant under $\Res(\mu,A)$ and hence, so is its closure $I$. Since the resolvent is analytic, it follows from the identity theorem for analytic functions (applied to the quotient space $L^\infty/I$) that $I$ is actually invariant under $\Res(\mu,A)'$ for every $\mu > 0$.
	
	But $\Res(\lambda,A)'$ is injective, so $\Res(\lambda,A)'f \not= 0$ and hence, assumption~(v) implies that there exists a number $\mu > 0$ such that $0 \ll \Res(\mu,A)' \Res(\lambda,A)'f \in I$. Therefore $\one \in I$, and we conclude that $I = L^\infty$.
	
	``(vi) $\Rightarrow$ (iii)'' The operator $\Res(1,A)'$ is positive, has spectral radius $1$ and is irreducible according to~(vi). Hence, it follows from~\cite[Corollary~2 on p.\,151 and Proposition~1 on p.\,146]{Lotz1981} that $1$ is a pole of the resolvent of $\Res(1,A)'$ and consequently, $1$ is also a pole of the resolvent of $\Res(1,A)$. Therefore, $0$ is a pole of the resolvent of $A$ \cite[Proposition~IV.1.18]{Engel2000}.
\end{proof}

We note that many of the implications in Corollary~\ref{cor:spectral-condition} remain correct without the assumption that $T_{t_0}$ dominates a non-trivial compact or integral operator. Of all implications that occurred in the proof above, only  the implication ``(iii) $\Rightarrow$ (i)'' needs this assumption. However, we do not know whether one can get from (ii), (iii), (v) or (vi) to (iv) without the detour via~(i).

\section{Outlook} \label{section:outlook}

Besides potential applications to mathematical biology and kinetic equations, Theorem~\ref{thm:main} might prove useful in the study of queuing systems. It is a common approach to analyse queuing systems and reliability models by means of stochastic $C_0$-semigroups on $L^1(\Omega,\mu)$-spaces; see for instance \cite{Gupur2011, Zheng2015, Haji2013, Gwizdz2019} for a few examples. In many cases the underlying measure space $(\Omega,\mu)$ contains an \emph{atom} (i.e., a measurable one point set $\{\omega\}$ with measure $0 < \mu(\{\omega\}) < \infty$), and as has been pointed out in \cite[Section~6.1]{Gerlach2019} the existence of such an atom implies that every irreducible stochastic semigroup on $(T_t)_{t \in [0,\infty)}$ contains a partial integral operator. Thus, Theorem~\ref{thm:main} shows that such a semigroup converges if and only if $0$ is a pole of the resolvent of the generator.

Another application is to show that finite networks flows, as for instance discussed in \cite{Kramar2005, Dorn2010}, automatically converge to an equilibrium if one introduces a mass buffer in one of the vertices. Such a result has been proved by the second named author in his Master's thesis \cite{Martin2018} and will be presented in a separate paper.

\subsection*{Acknowledgements}

A special case of Theorem~\ref{thm:main} was proved while the second-named author wrote his Master's thesis at the Institute of Applied Analysis at Ulm University in winter 2017/18 \cite{Martin2018}.

\bibliographystyle{plain}
\bibliography{literature}

\end{document}